\documentclass[11pt]{article}
\usepackage{amssymb,amsmath}
\usepackage{amsmath}
\usepackage{amsfonts}
\newtheorem{precor}{{\bf Corollary}}

\newenvironment{cor}{\begin{precor}{\hspace{-0.5
               em}{\bf.\ }}}{\end{precor}}
\newtheorem{precon}{{\bf Conjecture}}

\newenvironment{con}{\begin{precon}{\hspace{-0.5
               em}{\bf.\ }}}{\end{precon}}
\newtheorem{predefin}{{\bf Definition}}

\newtheorem{preexm}{{\bf Example}}

\newtheorem{preappl}{{\bf Application}}

\newtheorem{prelem}{{\bf Lemma}}

\newenvironment{lem}{\begin{prelem}{\hspace{-0.5
               em}{\bf.\ }}}{\end{prelem}}
\newtheorem{preproof}{{\bf Proof.\ }}

\newenvironment{proof}[1]{\begin{preproof}{\rm
               #1}\hfill{$\blacksquare$}}{\end{preproof}}
\newtheorem{presproof}{{\bf Sketch of Proof.\ }}

\newtheorem{prethm}{{\bf Theorem}}

\newenvironment{thm}{\begin{prethm}{\hspace{-0.5
               em}{\bf.\ }}}{\end{prethm}}
\newtheorem{prealphthm}{{\bf Theorem}}

\newtheorem{prealphlemma}{{\bf Lemma}}

\newtheorem{prepro}{{\bf Proposition}}

\newtheorem{preprb}{{\bf Problem}}

\newtheorem{prequ}{{\bf Question}}

\def\conct[#1,#2]{\mbox {${#1} \leftrightarrow {#2}$}}
\def\dconct[#1,#2]{\mbox {${#1} \rightarrow {#2}$}}
\def\deg[#1,#2]{\mbox {$d_{_{#1}}(#2)$}}
\def\mindeg[#1]{\mbox {$\delta_{_{#1}}$}}
\def\maxdeg[#1]{\mbox {$\Delta_{_{#1}}$}}
\def\outdeg[#1,#2]{\mbox {$d_{_{#1}}^{^+}(#2)$}}
\def\minoutdeg[#1]{\mbox {$\delta_{_{#1}}^{^+}$}}
\def\maxoutdeg[#1]{\mbox {$\Delta_{_{#1}}^{^+}$}}
\def\indeg[#1,#2]{\mbox {$d_{_{#1}}^{^-}(#2)$}}
\def\minindeg[#1]{\mbox {$\delta_{_{#1}}^{^-}$}}
\def\maxindeg[#1]{\mbox {$\Delta_{_{#1}}^{^-}$}}
\def\isdef{\mbox {$\ \stackrel{\rm def}{=} \ $}}
\def\dre[#1,#2,#3]{\mbox {${\cal E}_{_{#3}}(#1,#2)$}}
\def\pdre[#1,#2,#3]{\mbox {${\cal P}_{_{#3}}(#1,#2)$}}
\def\var[#1,#2]{\mbox {${\rm Var}_{_{#1}}(#2)$}}
\def\ls[#1]{\mbox {$\xi^{^{#1}}$}}
\def\hom[#1,#2]{\mbox {${\rm Hom}({#1},{#2})$}}
\def\onvhom[#1,#2]{\mbox {${\rm Hom^{v}}(#1,#2)$}}
\def\onehom[#1,#2]{\mbox {${\rm Hom^{e}}(#1,#2)$}}
\def\core[#1]{\mbox {$#1^{^{\bullet}}$}}
\def\cay[#1,#2]{\mbox {${\rm Cay}({#1},{#2})$}}
\def\cays[#1,#2]{\mbox {${\rm Cay_{s}}({#1},{#2})$}}
\def\dirc[#1]{\mbox {$\stackrel{\rightarrow}{C}_{_{#1}}$}}
\def\cycl[#1]{\mbox {${\bf Z}_{_{#1}}$}}

\def\sdg[#1]{\mbox {$\stackrel{\leftrightarrow}{#1}$}}
\begin{document}
\begin{center}
{\Large \bf Dynamic Chromatic Number of Regular Graphs}\\
\vspace*{0.5cm}
{\bf Meysam Alishahi}\\
{\it Department of Mathematics}\\
{\it Shahrood University of Technology, Shahrood, Iran}\\
{\tt meysam\_alishahi@shahroodut.ac.ir}\\
\end{center}
\begin{abstract}
\noindent A dynamic coloring of a graph $G$ is a proper coloring
 such that for every vertex $v\in V(G)$ of degree at least 2,
the neighbors of $v$ receive at least 2 colors.
It was conjectured [B.~Montgomery. {\em Dynamic coloring of graphs}.
PhD thesis, West Virginia University, 2001.] that if $G$ is a $k$-regular graph, then $\chi_2(G)-\chi(G)\leq 2$.
In this paper, we prove that if $G$ is a $k$-regular graph with $\chi(G)\geq 4$, then
$\chi_2(G)\leq \chi(G)+\alpha(G^2)$. It confirms the conjecture for all regular graph $G$ with diameter at most $2$ and $\chi(G)\geq 4$.
In fact, it shows that $\chi_2(G)-\chi(G)\leq 1$ provided that $G$ has diameter at most $2$ and $\chi(G)\geq 4$.
Moreover, we show that for any $k$-regular graph $G$, $\chi_2(G)-\chi(G)\leq 6\ln k+2$.
Also, we show that for any $n$ there exists a regular graph $G$ whose chromatic number is $n$ and $\chi_2(G)-\chi(G)\geq 1$.
This result gives a negative answer to a conjecture of [A.~Ahadi, S.~Akbari, A.~Dehghan, and M.~Ghanbari.
\newblock On the difference between chromatic number and dynamic chromatic
number of graphs. \newblock {\em Discrete Math.}, In press].
\begin{itemize}
\item[]{{\footnotesize {\bf Key words:}\  Dynamic chromatic number, $2$-colorability of hypergraphs.}}
\item[]{ {\footnotesize {\bf Subject classification: 05C} .}}
\end{itemize}
\end{abstract}
\section{Introduction}
Let $H$ be a hypergraph. The vertex set and the
hyperedge set of $H$  are mentioned as $V(H)$ and
$E(H)$, respectively. The maximum degree and the minimum
degree of $H$ are denoted by $\Delta(H)$ and
$\delta(H)$, respectively.
For an integer $l\geq 1$, denote by $[l]$, the set $\{1, 2,
\ldots, l\}$. A proper $l$-coloring of a hypergraph
$H$ is a function $c: V(H)\longrightarrow [l]$  in which there is no monochromatic
hyperedge in $H$. We say a hypergraph $H$ is
$t$-colorable if there is a proper $t$-coloring of it.  For a
hypergraph $H$, the smallest integer $l$ so that
$H$ is $l$-colorable is called the chromatic number of
 $H$ and denoted by $\chi(H)$. Note that a graph $G$ is a
hypergraph such that  the cardinality of each $e\in E(G)$ is 2.

A proper vertex $l$-coloring of a graph $G$ is called a dynamic
$l$-coloring \cite{montgomery} if for every vertex $u$ of degree
at least $2$, there are at least two different colors appearing in
the neighborhood of $v$. The smallest integer $l$ so that there is
a dynamic $l$-coloring of $G$ is called {\it the dynamic
chromatic number of $G$} and denoted by $\chi_2(G)$. Obviously,
$\chi(G)\leq\chi_2(G)$. Some properties of dynamic coloring were
studied in \cite{akbari2,MR2746973,MR2251583,MR1991048,montgomery,123}.
It was proved in \cite{MR1991048} that for a connected graph $G$ if  $\Delta\leq3$,
then $\chi_2(G)\leq4$ unless $G=C_5$, in which case
$\chi_2(C_5)=5$ and if $\Delta\geq4$, then $\chi_2(G)\leq\Delta+1$.
It was shown in \cite{montgomery} that the difference between chromatic number and
dynamic chromatic number can be arbitrarily large. However, it was
conjectured that for regular graphs the difference is at most 2.
\begin{con}{\rm\cite{montgomery}}\label{conj}
For any regular graph $G$, $\chi_2(G)-\chi(G)\leq 2$
\end{con}
Also, it was proved in \cite{montgomery} that if $G$ is a
bipartite $k$-regular graph, $k\geq 3$ and $n<2^k$, then
$\chi_2(G)\leq 4$. This result was extended to all regular bipartite graphs in \cite{MR2746973}.

In a graph $G$, a set $T\subseteq V(G)$ is called a {\it total
dominating set} in $G$ if for every vertex $v \in V(G)$, there is
at least one vertex $u \in T$ adjacent to $v$. The set $T\subsetneq V(G)$ is
called a {\it double total dominating set} if $T$ and its
complement $V(G)\setminus T$ are both total dominating \cite{MR2746973}.
Also, by ${\cal I}(G)$ and ${\cal IM}(G)$ we refer to the set of independent and maximal independent sets in $G$, respectively.

\section{results}
The $2$-colorability of hypergraphs has been studied in the literature
and has lots of applications in the other area of combinatorics.
\begin{thm}{\rm \cite{mcdiarmid}}\label{hyper}
Let $H$ be a hypergraph in which every hyperedge contains at least $k$ points
and meets at most $d$ other hyperedges. If $e(d + 2)\leq 2^k$, then $H$ is $2$-colorable.
\end{thm}

Assume that $G$ is a graph. Let $T\subseteq V(G)$ and define a hypergraph $H_G(T)$ whose vertex set is $\bigcup_{v\in T}N(v)$
and its  hyperedge set is defined as follows
$$E\left(H_G(T)\right)\isdef\left\{N(v)|v\in T \right\}.$$
Clearly, for any $f\in E(H_G(T))$, $\delta(G)\leq |f|\leq \Delta(G)$
and $\Delta(H_G(T))\leq\Delta(G)$. Therefore $f$ meets
at most $\Delta(G)(\Delta(G) -1)$ other hyperedges.

It was shown by Thomassen \cite{thomassen} that for any $k$-uniform and
$k$-regular hypergraph $H$, if $k\geq 4$, then $H$ is 2-colorable.
This result can be easily extended to all $k$-uniform hypergraphs with the maximum
degree at most $k$ \cite{MR2746973}, i.e.,
any $k$-uniform hypergraph $H$ with $k\geq4$ and the maximum
degree at most $k$, is 2-colorable.
By considering Theorem \ref{hyper}, if $e(\Delta^2(G)-\Delta(G)+2)\leq2^{\delta(G)}$ (in the $k$-regular case, $k\geq 4$), then $H_G(T)$
is 2-colorable.

Next lemma is proved in \cite{akbariliaghat} and extended to circular coloring in \cite{metath}.
\begin{lem}{\rm \cite{akbariliaghat}}\label{akbari2}
Let $G$ be a connected graph, and let $c$ be a $\chi(G)$-coloring of $G$. Moreover, assume that
$H$ is a nonempty subgraph of $G$. Then there exits a $\chi(G)$-coloring $f$ of $G$ such that:
\begin{enumerate}
\item[a{\rm)}] if $v\in V(H)$, then $f(v)=c(v)$ and
\item[b{\rm)}] for every vertex $v\in V(G)\setminus V(H)$ there is a
 path $v_0v_1 \ldots v_m$ such that $v_0=v$, $v_m$ is in $H$, and $f(v_{i+1})-f(v_i)=1\ \pmod{\chi(G)}$ for $i=0,1, \ldots, m-1$.
\end{enumerate}
\end{lem}

Assume that $G$ is a given graph. The graph $G^2$ is a graph with the vertex set $V(G)$ and two different
vertices $u$ and $v$ are adjacent in $G^2$ if $d_G(u,v)\leq 2$, i.e., 
there is a walk with length at most two between $u$ and $v$ in $G$.

The connection between the independence number and the dynamic chromatic 
number of graphs has been studied in \cite{ahdiak}. 
The first part of the next theorem improves a similar result in \cite{ahdiak} and
in the second part of it we present an upper bound for the dynamic chromatic number of 
graph $G$ in terms of chromatic number of $G$ and the independence
number of $G^2$.
\begin{thm}\label{newthm}
\begin{enumerate}
\item[{\rm 1)}] For any graph $G$ with $\chi(G)\geq 4$, $\chi_2(G)\leq \chi(G)+\alpha(G)$.
\item[{\rm 2)}] If $G$ is a $k$-regular graph with $\chi(G)\geq 4$, then
$\chi_2(G)\leq \chi(G)+\alpha(G^2)$.
\end{enumerate}
\end{thm}
\begin{proof}{
Let $uv$ be an edge of $G$. Assume that $c$ is a $\chi(G)$-coloring of $G$ such that $c(u)=1$ and $c(v)=3$.
Let $H$ be the subgraph induced on the edge $uv$ and $f$ be a coloring as in Lemma~\ref{akbari2}.
According to Lemma \ref{akbari2}, $f(u)=1$ and $f(v)=3$. We call a vertex $v\in V(G)$, a {\it bad} vertex if
${\rm deg}(v)\geq 2$ and all vertices in $N(v)$ have the same color in $f$.
Let $S$ be the set of all bad vertices. We claim that $S$ is an independent set in $G$.
Suppose therefore (reductio ad absurdum) that this is not the case. We consider four different cases.
\begin{enumerate}
\item $\{u,v\}\subset S$. Since $u$ and $v$ are both bad vertices and $uv\in E(G)$, all the vertices in $N(u)$ have the color $3$ and all the vertices in $N(v)$ have the color $1$. Note that $\chi(G)\geq 4$. Therefore the coloring $f$ does~not satisfy Lemma \ref{akbari2} and it is a contradiction.

\item There exists a vertex $x\neq v$ such that $\{u,x\}\subseteq S$ and $ux\in E(G)$. Since $v\in N(u)$ and $u\in N(x)$ and both $u$ and $x$ are bad, all the vertices in $N(u)$ and $N(x)$ have the colors $3$ and $1$ in the coloring $f$, respectively. According to Lemma \ref{akbari2}, there is a path $v_0v_1\ldots v_m$ in $G$ such that $v_0=x$, $v_m\in V(H)$, and
      $f(v_{i+1})-f(v_i)=1\ \pmod{\chi(G)}$ for $i=0,1, \ldots, m-1$. But, this is~not possible, because all the neighbors of $x$ have the color $1$ and $x$ has the color $3$.

\item There exists a vertex $y\neq u$ such that $\{v,y\}\subseteq S$ and $vy\in E(G)$. This case is the same as the previous case.

\item There are two vertices $x$ and $y$ in $S\setminus \{u,v\}$ such that $xy\in E(G)$.
      Note that according to Lemma~\ref{akbari2}, there should be at least two vertices $z\in N(x)$ and $z'\in N(y)$ such that $f(z)=f(x)+1$ and $f(z')=f(y)+1$ $({\rm mod} \chi)$. Since $x$ and $y$ are both bad vertices,
      all the vertices in $N(x)$ have the color $f(y)$ and all the vertices  in $N(y)$ have the color $f(x)$.
      Therefore $f(z)=f(y)$ and $f(z')=f(x)$ ${\rm( mod}\ \chi(G){\rm)}$, but this is~not possible because $\chi(G)\geq 4$.
\end{enumerate}
Now, we know $S$ is an independent set in $G$ and so $|S|\leq \alpha(G)$. For any vertex $w\in S$,
choose a vertex $x(w)\in N(w)$ and put all these vertices in $S'$. Assume that $S'=\{x_1,x_2,\ldots,x_t\}$. Consider a coloring $f'$ such that
$f$ and $f'$ are the same on $V(G)\setminus S'$  and for any $x_i\in S'$,
$x_i$ is colored with $\chi(G)+i$. One can easily check that $f'$ is a dynamic coloring of $G$ used at most $\chi(G)+\alpha(G)$ colors.

To prove the second part, assume that $G$ is a $k$-regular graph with $\chi(G)\geq 4$.
Consider the coloring $f$ and the set $S$ as in the previous part.
Assume that $G^2[S]$, i.e., the induced subgraph of $G^2$ on the vertices in $S$, has the components $G^2_1,G^2_2,\ldots,G^2_n$.
Note that two different vertices $x,y\in S$ are adjacent in $G^2$ if and only if $N_G(x)\cap N_G(y)\neq \varnothing$
(since $S$ is an independent set, $xy\not\in E(G)$).
Therefore for any $1\leq i\leq n$, all the vertices in
$$N_i=\bigcup_{x\in V(G^2_i)}N_G(x)$$
have the same color in the coloring $f$. For any $1\leq i\leq n$,
let $H_i$ be a hypergraph with the vertex set $N_i$ and with the hyperedge set
$$E(H_i)=\{N(x)|\ x\in V(G^2_i)\}.$$ It is clear that $H_i$ is a $k$-uniform hypergraph with $\Delta(H_i)\leq k$.
Since $\chi(G)\geq 4$, we have $k\geq 4$ or $G=K_4$. If $G=K_4$, then $\chi_2(G)=4\leq\chi(G)+\alpha(G^2)$ and there is nothing to prove.
Now, we can assume that $k\geq 4$.
According to the discussion after Theorem \ref{hyper}, $H_i$ is $2$-colorable.
For any $1\leq i\leq n$, let $(X^1_i,X^2_i)$ be a $2$-coloring of $H_i$.
Define $f''$ to be a coloring of $G$ such that $f''$ and $f$ are the same on $V(G)\setminus(\bigcup X^1_i)$ and for each
$1\leq i\leq n$, $f''$ has the constant value $i+\chi(G)$ on the vertices of $X^1_i$.
It is easy to see that $f''$ is a $(\chi(G)+n)$-dynamic coloring of $G$. Obviously, $n\leq \alpha(G^2)$ and the proof is completed.
}\end{proof}
In the proof of the second part of Theorem~\ref{newthm}, we need the $2$-colorability of all $H_i$'s and if some assumptions cause this property,
then the remain of proof still works. Consequently, in view of the discussion after Theorem \ref{hyper}, we have the next corollary.
\begin{cor}
Let $G$ be a graph such that $\chi(G)\geq 4$ and $e(\Delta^2(G)-\Delta(G)+1)\leq2^{\delta(G)}$. Then $\chi_2(G)\leq \chi(G)+\alpha(G^2).$
\end{cor}
{\bf Remark.} Note that in the proof of Theorem \ref{newthm}, it is shown that for any $k$-regular graph
$G$ with $\chi(G)\geq 4$, $\chi_2(G)\leq \chi(G)+ {\rm com}(G^2[S])$
where ${\rm com}(G^2[S])$ is the number of connected components of $G^2[S]$ and $S$ is an
independent set given in the proof of Theorem \ref{newthm}. Therefore for any graph $G$
with $\chi(G)\geq 4$ and $e(\Delta^2(G)-\Delta(G)+1)\leq2^{\delta(G)}$ {\rm (}in $k$-regular case $k\geq 4${\rm )},
$$\chi_2(G)\leq \chi(G)+\displaystyle\max_{I\in {\cal I}(G)}{\rm com}(G^2[I]).$$

It is shown in \cite{akbariars} that if $G$ is a strongly regular graph except $C_5$ and $K_{m,m}$,
then $\chi_2(G)-\chi(G)\leq 1$. Note that for a graph $G$ with diameter $2$,
the graph $G^2$ is a complete graph and $\alpha(G^2)=1$. Therefore the second
part of Theorem \ref{newthm} extends this result to a larger family of regular graphs.
In fact, every strongly regular graph has diameter at most 2, but
according to the second part of Theorem~\ref{newthm}, if $G$ is a $k$-regular
graph with diameter $2$ and $\chi(G)\geq 4$, then $\chi_2(G)-\chi(G)\leq 1$.
Moreover, by the previous corollary, if $G$ is a graph with diameter $2$, $\chi(G)\geq 4$
and $e(\Delta^2(G)-\Delta(G)+1)\leq2^{\delta(G)}$, then $\chi_2(G)-\chi(G)\leq 1$.
We restate this result in the next corollary.
\begin{cor}
Let $G$ be a graph with diameter $2$, $\chi(G)\geq 4$ and $e(\Delta^2(G)-\Delta(G)+1)\leq2^{\delta(G)}$ {\rm (}in $k$-regular case $k\geq 4${\rm )}.
Then $\chi_2(G)-\chi(G)\leq 1$
\end{cor}

The proof of the second part of Theorem \ref{newthm} strongly depended on the assumption that $\chi(G)\geq 4$. In fact the only bipartite regular graphs with diameter 2 are complete regular bipartite graphs whose chromatic number and dynamic chromatic number are $2$ and $4$, respectively. But $K_{m,m}^2$ is a complete graph and so $\chi(K_{m,m})+\alpha(K_{m,m})=3<\chi_2(K_{m,m})$. For the case of $\chi(G)=3$, if we set $G=C_5$, then
$C_5^2=K_5$ and $\chi_2(C_5)=5>\chi(C_5)+\alpha(C_5^2)$.

Note that in the proof of Theorem~\ref{newthm}, we assumed that $\chi(G)\geq 4$ because
we want to use Lemma \ref{akbari2} to obtain a coloring $f$ such that all the bad vertices related to $f$ form an independent set in $G$.
However, if one finds a $t$-coloring $f$ of $G$ such that the set of bad vertices related to $f$, $S$, is an independent set in $G$, then $\chi_2(G)\leq t+\alpha(G)$ and if $G$ is a $k$-regular graph with $k\geq 4$, then $\chi_2(G)\leq t+{\rm com}(G^2[S])$.

Now, let $G$ be a graph such that $e(\Delta^2(G)-\Delta(G)+1)\leq2^{\delta(G)}$ {\rm (}in $k$-regular case, $k\geq 4${\rm )}
and let $I$ be an arbitrary maximal independent set in $G$. Consider an optimum $t$-coloring $c$ of $G$ such that $I$ is a
color class in this coloring($t$ is the least possible number).
Define $H$ to be a hypergraph with vertex set $I$ and the hyperedge set $E(H)=\{N(v)|\ v\in V(G)\ \&\ N(v)\subseteq I\}$.
Since $e(\Delta^2(G)-\Delta(G)+1)\leq2^{\delta(G)}$ (in $k$-regular case, $k\geq 4$), $H$ is $2$-colorable.
Let $(X,Y)$ be a $2$-coloring of $H$. Recolor the vertices in $Y$ with a new color $t+1$ to obtain a $(t+1)$-coloring $f$ of $G$.
It is readily seen that $S$, the set of the bad vertices related to $f$, is a subset of $I$ and therefore it is an independent set.
By the same argument as in the proof of the second part of Theorem \ref{newthm}, one can show that $\chi_2(G)\leq t+1+{\rm com}(G^2[S])$.
Now, note that $t\leq \chi(G)+1$ and so $\chi_2(G)\leq t+1+{\rm com}(G^2[S])\leq \chi(G)+2+\displaystyle\max_{P\subseteq I}{\rm com}(G^2[P])$.
Let ${\cal IM}(G)$ be the set of all maximal independent sets in $G$.
Since $I$ is an arbitrary maximal independents set in $G$,
$$\chi_2(G)\leq \chi(G)+\displaystyle\min_{I\in{\cal IM}(G)}\max_{P\subseteq I}{\rm com}(G^2[P])+2.$$

In Theorem \ref{newthm}, we have the assumption $\chi(G)\geq 4$ and for a graph $G$ with $\chi(G)<4$, we can not use this theorem.
In view of the above discussion, if we consider $c$ as a $\chi(G)$-coloring of $G$
such that the color class $V_1$ (all the vertices with color $1$) is a maximal independent set in $G$($t=\chi(G)$), then $\chi_2(G)\leq \chi(G)+\alpha(G)+1$ and also, we have the next corollary.
\begin{cor}
Let $G$ be a graph such that $e(\Delta^2(G)-\Delta(G)+1)\leq2^{\delta(G)}$ {\rm (}in $k$-regular case, $k\geq 4${\rm )}.
Then $\chi_2(G)\leq \chi(G)+\alpha(G^2)+1.$
\end{cor}

Erd\H os and Lov{\'a}sz \cite{ELlocallemma} proved a very powerful lemma, known as the Lov{\'a}sz Local Lemma.\\
{\bf The Lov{\'a}sz Local Lemma}.\cite{MR2437651}
Let $A_1,A_2,\ldots,A_n$ be evens in an arbitrary probability space. Suppose that each event $A_i$
is mutually independent of a set of all  $A_j$ but at most $d$ of the other events and ${\rm Pr}(A_i)\leq p$
for all $1\leq i\leq n$. If $ep(d+1)\leq 1$ then $\displaystyle{\rm Pr}(\bigcap_{i=1}^n\bar{A_i})>0$

It was proved in \cite{MR2746973} that for any $k$-regular graph $G$, the difference between dynamic chromatic number and chromatic number of $G$ is at most $14.06\ln k +1$. In the next theorem we shall improve this result.
\begin{thm}
For any $k$-regular graph $G$, $\chi_2(G)-\chi(G)\leq 6\ln k+2$.
\end{thm}
\begin{proof}{
It is proved in \cite{MR2746973}, that for any regular graph $G$, $\chi_2(G)\leq 2\chi(G)$.
Therefore for any $k$-regular graph $G$ with $k\leq 3$, $\chi_2(G)\leq 6\leq 6\ln k+2$.
Now, we can assume that $k\geq 4$.
Let $V(G)=\{v_1,v_2,\ldots,v_n\}$. For any permutation (total ordering) $\sigma\in S_{V(G)}$, set
$$I_\sigma=\left\{v\in V(G)|\ v\prec_{\sigma}u\ {\rm for\ all}\ u\in N(v)\right\}.$$
It is readily seen that $I_\sigma$ is an independent set of $G$.
Assume that $U\subseteq V(G)$ consists of all vertices that are~not
lied in any triangle. Now, choose $l$ permutations $\sigma_1,\ldots,\sigma_l$,
randomly and independently. For any $u\in U$, let $A_u$ be the event that
there are~not a vertex $v\in N(u)$ and $\sigma_i$ such that the vertex $v$ precedes
all of its neighbors in the permutation $\sigma_i$, i.e., $N(u)\bigcap \displaystyle\bigcup_{i=1}^lI_{\sigma_i}=\varnothing$.
Since $u$ dose not appear in any triangle,
one can easily see that ${\rm Pr}(A_u)=(1-{1\over k})^{kl}\leq e^{-l}$.
Note that $A_u$ is mutually independent of all events $A_v$ for which $d(u,v)> 3$.
Consequently, $A_u$ is mutually independent of all but at most $k ^3-k^2+k+1$ events.
In view of Lov{\'a}sz Local Lemma, if $e(k ^3-k^2+k+2)e^{-l}\leq 1$, then none of the events
$A_u$ happens with positive probability. In other words, for $l=\lceil 3\ln k+1 \rceil$,
there are permutations $\sigma_1,\ldots,\sigma_l$ such that $A_u$ does not
happen for any $u\in U$. It means that for any vertex $u\in U$, there is an $i$ ($1\leq i\leq l$)
such that $N(u)\cap I_{\sigma_i}\neq\varnothing$. Note that if we set $T=\displaystyle\bigcup_{i=1}^l I_{\sigma_i}$, then
$G[T]$, the induced subgraph on $T$, has the chromatic number at most $l$.
Assume that $c_1$ is a proper $\chi(G[T])$-coloring of $G[T]$. Now, let $H$
be a hypergraph with the vertex set $T$ and the hyperedge set defined as follows
$$E(H)=\{N(v)|v\in V(G),\ N(v)\subseteq T\}.$$
One can check that $H$ is $2$-colorable. If $H$ is an empty hypergraph, there is noting to prove.
Otherwise, since $H$ is a $k$-uniform hypergraph with maximum degree at most $k$ ($k\geq 4$),
$H$ is $2$-colorable. Assume that $c_2$ is a
$2$-coloring of $H$. It is obvious that $c=(c_1,c_2)$ is a $2l$-coloring of
$G[T]$. Now, consider a $(\chi(G)+2l)$-coloring $f$ for $G$ such that the
restriction of $f$ on $T$ is the same as $c$. One can check that $f$ is a dynamic coloring of $G$.
}\end{proof}
It is proved in \cite{MR2746973} that for any $c>6$, there is a threshold $n(c)$ such that if $G$ is a $k$-regular
graph with $k\geq n(c)$ then, $G$ has a total dominating set inducing a graph with maximum degree at most $2c\log k$
(for instance if we set $c=7.03$ then $n(c)\leq 139$).
Note that in the proof of previous theorem, it is proved that any triangle free $k$-regular graph $G$ has a total dominating set $T$ such that
the induced subgraph $G[T]$ has the chromatic number at most $\lceil3\ln k +1\rceil$.

In the rest of the paper by $G\times H$ and $G\square H$, we refer to the Categorical product and
Cartesian product of graphs $G$ and $H$, respectively.
It is well-known that if $G$ is a graph with $\chi(G)>n$, then $G\times K_n$ is a uniquely $n$-colorable graph, see \cite{godsil2001algebraic}.

It was conjectured in \cite{ahdiak} that for any regular graph $G$ with $\chi(G)\geq 4$,
the chromatic number and the dynamic chromatic number are the same.
Here, we present a counterexample for this conjecture.

\begin{prepro}\label{counterexmp}
For any integer $n>1$, there are regular graphs with chromatic number $n$ whose dynamic chromatic number
is more than $n$.
\end{prepro}
\begin{proof}{
Assume that $G_1$ is a $d$-regular graph with $\chi(G_1)>n$ and $m=|V(G_1)|$.
Set $G_2= G_1 \square C_{(n-1)(d+2)+1}$ and $G'=G_2\times K_n$.
Note that $G'$ is a uniquely $n$-colorable graph with regularity $(n-1)(d+2)$.
Consider the $n$-coloring $(V_1,V_2,\ldots,V_n)$ for $G'$, where for $1\leq i\leq n$, $V_i=\{(g,i)\ |\ g\in V(G_2)\}$.
It is obvious that $|V_i|=m((n-1)(d+2)+1))$ is divisible by $(n-1)(d+2)+1$.
Now, for each $1\leq i\leq n$, consider $(S^i_1,S^i_2,\ldots,S^i_m)$ as
a partition of $V_i$ such that $|S^i_j|=(n-1)(d+2)+1$.
Now, for any $1\leq i\leq n$ and $1\leq j\leq m$, add a new vertex $s_{ij}$
and join this vertex to all the vertices in $S^i_j$ to construct the graph $G$.
Note that $G$ is an $((n-1)(d+2)+1)$-regular graph with $\chi(G)=n$.
Now, we claim that $\chi_2(G)>n$. To see this, assume that $c$ is an $n$-dynamic coloring of $G$. Since $G'$ is uniquely $n$-colorable,
the restriction of $c$ to $V(G)$ is $(V_1,V_2,\ldots,V_n)$. In the other words, all the vertices in $V_1$ have the same colore in $c$.
But, all the neighbors of $s^1_1$ are in $V_1$ and this means that $c$ is not a dynamic coloring.
}\end{proof}
However, it can be interesting to find some regular graph $G$ with $\chi(G)\geq 4$ and $\chi_2(G)-\chi(G)\geq 2$.

Also, as a generalization of Conjecture \ref{conj}, it was conjectured \cite{ahdiak} that for any graph $G$, $\chi_2(G)-\chi(G)\leq \lceil \frac{\Delta(G)}{\delta(G)}\rceil+1$. Here, we give a negative answer to this conjecture. To see this, assume that $G_1$ is a graph with $\chi(G_1)\geq 3$ and $n$ vertices such that $n>3\Delta(G_1)+5$. For any $2$-subset $\{u,v\}\subseteq V(G_1)$, add a new vertex $x_{uv}$ and join this vertex to the vertices $u$ and $v$.
Let $G$ be the resulting graph form $G_1$ by using this construction.
Note that $\chi_2(G)\geq n$, $\Delta(G)=\Delta(G_1)+n-1$, $\delta(G)=2$ and $\chi(G)=\chi(G_1)$.
Therefore if the conjecture was true, then we would have $$n-\Delta(G_1)-1\leq n-\chi(G_1)\leq \chi_2(G)-\chi(G)\leq \lceil \frac{\Delta(G_1)+n-1}{2}\rceil+1.$$
Note that it is not possible because $n>3\Delta(G_1)+5$, and consequently, $n-\Delta(G_1)-1>\lceil \frac{\Delta(G_1)+n-1}{2}\rceil+1$.\\

\noindent{\bf Acknowledgment}\\
The author wishes to express his deep gratitude to Hossein Hajiabolhassan
for drawing the author's attention to two counterexamples (Proposition~\ref{counterexmp} and the example after that).

\end{document}